\documentclass{amsart}
\usepackage{amsmath, amssymb,epic,graphicx,mathrsfs,enumerate}
\usepackage[all]{xy}
\usepackage{amsaddr}
\usepackage{mathtools}
\usepackage{amsmath,amsxtra,amssymb,latexsym, amscd,amsthm,times, float}
\usepackage{indentfirst}

\restylefloat{table}
\usepackage{longtable}

\restylefloat{table}

\usepackage[all]{xy}
\usepackage{eepic}
\usepackage[numbers]{natbib}
\usepackage{longtable}
\restylefloat{table}
\usepackage{multirow}
\usepackage{amsmath,amsxtra,amssymb,latexsym, amscd}
 \usepackage[mathscr]{eucal}

 \usepackage{graphicx}

\restylefloat{table}

\usepackage{latexsym}
\usepackage{longtable}
\usepackage{epsfig}
\usepackage{amsmath}
\usepackage{hhline}
\def\l{\langle}
\def\r{\rangle}
\def\a{\alpha}
\def\b{\beta}
\def\bb{\mathbb}
\def\cal{\mathcal}
\def\rm{\textrm}
\def\bf{\textbf}

\def\F{\mathrm{Frat}}
\def\soc{\mathrm{soc}}
\def\aut{\mathrm{Aut}}

\DeclareMathOperator{\frat}{Frat}

\begin{document}
\bibliographystyle{amsplain}

\title[Rational probabilistic zeta function]{Rationality of the probabilistic zeta functions\\ of finitely generated
profinite groups}
\author{Duong Hoang Dung}
\address{
Dipartimento di Matematica, Universit\`{a} degli studi di Padova,\\
Via Trieste 63, 35121 Padova, Italy\\
Mathematisch Instituut, Leiden Universiteit,\\ Niels Bohrweg 1, 2333 CA Leiden, The Netherlands}
\email{dhdung1309@gmail.com}
\author{Andrea Lucchini}
\address{Dipartimento di Matematica, Universit\`{a} degli studi di Padova,\\
Via Trieste 63, 35121 Padova, Italy.}
\email{lucchini@math.unipd.it}
\thanks{Research partially supported
by MIUR-Italy via PRIN "Group theory and applications"}
\subjclass[2010]{20E18, 20D06, 20P05, 11M41.}

\maketitle

\begin{abstract}
We prove that if the probabilistic zeta function  $P_G(s)$  of a finitely generated profinite group $G$ is rational and all but finitely many nonabelian composition factors of $G$ are groups of Lie type in a fixed characteristic or sporadic simple groups, then $G$ contains only finitely many maximal subgroups.
\end{abstract}

\newcommand\node[2]{\overset{#1}{\underset{#2}{\circ}}}
\newcommand{\DEF}[1]{{\em #1\/}}

\numberwithin{equation}{section}

\newtheorem{theorem}{\bf Theorem}[section]
\newtheorem{corollary}[theorem]{\bf Corollary}
\newtheorem{lemma}[theorem]{\bf Lemma}
\newtheorem{example}[theorem]{\bf Example}
\newtheorem{proposition}[theorem]{\bf Proposition}
\newtheorem{conjecture}[theorem]{\bf Conjecture}
\newtheorem{problem}[theorem]{\bf Problem}
\newtheorem{remark}[theorem]{\bf Remark}
\newtheorem{definition}[theorem]{\bf Definition}
\newtheorem*{Main1}{Theorem 1}
\newtheorem*{Main2}{Theorem 2}
\newtheorem*{Main3}{Corollary}
\newtheorem{claim}{\bf Claim}
\newtheorem*{my2}{Theorem \ref{my2}}
\newtheorem*{my3}{Theorem \ref{my3}}

\newcommand{\Prob}{\operatorname{Prob}}

\
\section{Introduction}

Let $G$ be a finitely generated profinite group. As $G$ has only
finitely many open  subgroups of a given index, for any $n \in
\mathbb N$ we may define the integer $a_n(G)$ as $a_n(G)=\sum_H
\mu_G(H),$ where the sum is over all open subgroups $H$ of $G$
with $|G:H|=n.$ Here $\mu_G(H)$ denotes the M\"obius function of
the poset of open subgroups of $G,$ which is defined by recursion
as follows: $\mu_G(G)=1$ and $\mu_G(H)=-\sum_{H<K}\mu_G(K)$ if
$H<G$.
 Then
we  associate to $G$ a formal Dirichlet series $P_G(s)$,  defined as
\[P_G(s)=\sum_{n \in \mathbb N}\frac{a_n(G)}{n^s}.\]

Hall in \cite{hall} showed that if $G$ is a finite group and $t$ is a positive integer, then $P_{G}(t)$ is equal to the probability that $t$ random elements of $G$ generate $G$ or in other words $$P_{G}(t)=\Prob
_{G}(t):=\frac{\left|\Omega_{G}(t)\right|}{|G^{t}|},$$ where $\Omega_{G}(t)$ is the set of  generating  $t$-tuples in $G$.
In \cite{pfg} Mann conjectured that $P_G(s)$ has a similar probabilistic meaning for a wide class of profinite groups. More precisely
define $\Prob_{G}(t)=\mu(\Omega_{G}(t))$,  where $\mu$ is the  normalised   Haar measure uniquely defined on the profinite group $G^{t}$ and $\Omega_{G}(t)$ is the set of generating $t$-tuples in $G$ (in the topological sense) and say that $G$ is positively finitely generated
if there exists a positive integer $t$ such that $\Prob_{G}(t)>0.$ Mann conjectured that if $G$ is positively finitely generated, then
$P_{G}(s)$ converges in some right half-plane and $P_{G}(t)=\Prob_{G}(t)$, when $t \in \Bbb N$ is large enough.
The second author proved in \cite{Lu11} that this conjecture is true if $G$ is a profinite group with polynomial subgroup growth.
But even when the convergence is not ensured, the formal Dirichlet series $P_G(s)$ encodes information about
the lattice generated by the maximal subgroups of $G$ and combinatorial properties of the probabilistic sequence $\{a_{n}(G)\}$ reflect
aspects of the structure of $G.$ For example in \cite{moltprof} it is proved that a finitely generated profinite group $G$ is prosolvable if and only if the sequence $\{a_{n}(G)\}$ is multiplicative. Notice that if $H$ is an open subgroup of $G$ and $\mu_G(H)\neq 0,$ then $H$ is an intersection of maximal subgroups of $G.$ This implies in particular that if $G$ contains only finitely many maximal subgroups (i.e. if the Frattini subgroup
$\frat G$ of $G$ has finite index in $G$), then there are only finitely many open subgroups $H$ of $G$ with $\mu_G(H)\neq 0$ and consequently $a_n(G)=0$ for all but finitely many $n\in \Bbb N$ (i.e. $P_G(s)$ is a finite Dirichlet series). A natural question is whether the converse is true.
An affirmative answer has been given in the case of prosolvable groups \cite{DeLu06}. Really a stronger result holds: if $G$ is a finitely generated prosolvable group, then $P_G(s)$ is rational (i.e. $P_G(s)=A(s)/B(s)$ with $A(s)$ and $B(s)$ finite Dirichlet series) if and only if $G/\frat G$ is a finite group. This has been generalized in \cite{DeLu07} to the finitely generated profinite groups with the property that all but finitely many factors in a composition series are either abelian or alternating groups. In this paper we prove two other results of the same nature.

\begin{Main1}\label{uno}
Let $G$ be a finitely generated  profinite group. Assume that there exist a prime $p$ and a normal open subgroup $N$ of $G$ such
that the nonabelian composition factors of $N$ are simple groups of Lie type over fields of characteristic $p$. Then $P_G(s)$ is rational if and only if $G/\F(G)$ is a finite group.
\end{Main1}

\begin{Main2}\label{due}
Let $G$ be a finitely generated  profinite group. Assume that there exists a normal open subgroup $N$ of $G$ such
that the nonabelian composition factors of $N$ are sporadic simple groups. Then $P_G(s)$ is rational if and only if $G/\F(G)$ is a finite group.
\end{Main2}

In particular, if $G$ contains a normal open subgroup $N$ all of whose
nonabelian composition factors are isomorphic, then we may apply the main theorem in \cite{DeLu06} if $N$ is prosolvable,
the main theorem in \cite{DeLu07} if $N$ has a composition factor of alternating type,
Theorem 1 if $N$ has a composition factor of Lie type and Theorem 2 if a sporadic simple groups
appears as a composition factor of $G$ and deduce the following corollary.

\begin{Main3}
Let $G$ be a finitely generated  profinite group. Assume that there exists a normal open subgroup $N$ of $G$ such
that the nonabelian composition factors of $N$ are all isomorphic. Then $P_G(s)$ is rational if and only if $G/\F(G)$ is a finite group.
\end{Main3}

The idea of the proof is the following. In \cite{DeLu03,DeLu06b} it is proved that $P_G(s)$ can be written
as formal product $P_G(s)=\prod_{i}P_i(s)$ of finite Dirichlet series associated with the non-Frattini factors in a chief series of $G$.
On the other hand $G/\frat(G)$ is finite if and only if a chief series of $G$ contains only finitely many non-Frattini factors. So the strategy
is to prove that the product $\prod_{i}P_i(s)$ cannot be rational if it involves infinitely many nontrivial factors. A consequence of
the Skolem-Mahler-Lech Theorem  (see Proposition \ref{skolem theorem}) can help us in this task. However Proposition \ref{skolem theorem} concerns infinite product
of finite Dirichlet series involving only one nontrivial summand, but only the finite Dirichlet series associated to the abelian chief factors
of $G$ have this property, while in general the finite series $P_i(s)$ are quite complicated. So we need to produce suitable  \lq\lq short\rq\rq  \ approximations $P_i^*(s)$ of the series $P_i(s)$, in such a way that the rationality of their product is preserved: the tool to
achieve such approximation is a slight modification of a result already employed in \cite{DeLu07} for a similar purpose (see Proposition \ref{prop 4.3}). This requires a delicate analysis of the subgroup structure of the
almost simple groups of Lie type, based in particular on the  properties of the parabolic subgroups, and some information on the maximal subgroups of the sporadic simple groups.

\section{Infinite products of formal Dirichlet series}
Let $\cal R$ be the ring of formal Dirichlet series with integer
coefficients. We will say that $F(s)\in \cal R$ is rational if there
exist two finite Dirichlet series $A(s), B(s)$ with $F(s)=A(s)/B(s).$

For every set $\pi$ of prime numbers, we consider the ring
endomorphism  of $\cal R$ defined by:
\begin{eqnarray*}
F(s)=\sum_{n\in \bb{N}} \frac{a_n}{n^s}  \mapsto  F^{\pi}(s)=\sum_{n\in \bb{N}} \frac{a_n^*}{n^s}
\end{eqnarray*}
where $a_n^*=0$ if $n$ is divisible by some prime $p\in \pi,$ $a_n^*=a_n$ otherwise.
We will use the following remark:
\begin{remark}\label{pirat}
For every set $\pi$ of prime numbers, if $F(s)$ is rational then $F^\pi(s)$ is rational.
\end{remark}

The following result is a consequence of the Skolem-Mahler-Lech Theorem:
\begin{proposition}\label{skolem theorem}\cite[Proposition 3.2]{DeLu06}
Let $I\subseteq \bb{N}$ and let $q,r_i,c_i$ be positive integers for each $i\in I$. Assume that
\begin{itemize}
\item[(i)] for every $n\in\bb{N}$, the set $\{i\in I \mid r_i~\rm{divides}~n\}$ is finite;
\item[(ii)] there exists a prime $t$ such that $t$ does not divide $r_i$ for any $i\in I.$
\end{itemize}
If the product
$$F(s)=\prod_{i\in I}\left(1-\frac{c_i}{(q^{r_i})^s}\right)$$
is rational, then $I$ is finite.
\end{proposition}

The following slight modification of Proposition 4.3 in \cite{DeLu07} can be proved exactly in the same way and will play a significant  role in our arguments.
\begin{proposition}\label{prop 4.3}
Let $F(s)$ be a product of finite Dirichlet series $F_i(s)$, indexed over a subset $I$ of $\mathbb N$:
$$F(s)=\prod_{i\in I}F_i(s),~\textrm{where}~ F_i(s)=\sum_{n\in\bb{N}}\frac{b_{i,n}}{n^s}$$
Let $q$ be a prime and $\Lambda$ the set of positive integers divisible by $q$. Assume that there exists a positive integer $\alpha$ and a set $\{r_i\}_{i\in I}$  of positive integers such that if $n\in\Lambda$ and $b_{i,n}\ne 0$ then $n$ is an $r_i$-th power of some integer and $v_q(n)=\a r_i$ (where $v_q(n)$ is the $q$-adic valuation of $n$).  Define
$$w=\min\{x\in\bb{N}~\big|~ v_q(x)=\a~\textrm{and}~b_{i,x^{r_i}}\ne 0~\textrm{for some}~i\in I\}.$$
If $F(s)$ is rational, then the product
$$
F^*(s)=\prod_{i\in I}\left(1+\frac{b_{i,w^{r_i}}}{(w^{r_i})^s}\right)
$$
is also rational.
\end{proposition}

\section{Preliminaries and notations}\label{notation}
Let $G$ be a finitely generated profinite group and let $\{G_i\}_{i\in\bb{N}}$  be a fixed countable descending series of open normal subgroups with the property that $G_0=G$, $\cap_{i\in\bb{N}}G_i=1$ and $G_i/G_{i+1}$ is a chief factor of $G/G_{i+1}$ for each $i\in\bb{N}$. In particular, for
each $i \in \bb{N}$, there exist a simple group $S_i$ and a positive integer $r_i$ such that $G_i/G_{i+1}\cong S_i^{r_i}.$
Moreover, as described in \cite{DeLu06b}, for each $i\in \bb{N}$, a finite
Dirichlet series
\begin{equation}\label{eq-pi}
P_i(s)=\sum_{n\in \bb{N}}\frac{b_{i,n}}{n^s}
\end{equation}
is associated with the chief factor  $G_i/G_{i+1}$ and $P_G(s)$
can be written as an infinite formal product of the finite Dirichlet series $P_i(s)$:
\begin{equation}P_G(s)=\prod_{i\in \bb{N}}P_i(s).
\end{equation}
Moreover, this factorization is independent of the choice of chief series (see \cite{DeLu03,DeLu06b}) and $P_i(s)=1$ unless $G_i/G_{i+1}$ is a non-Frattini chief factor of $G$.

We recall some properties of the series $P_i(s).$ If $S_i$ is cyclic of order $p_i,$
then $P_i(s)=1-c_i/(p_i^{r_i})^s,$
where $c_i$ is the number of complements of $G_i/G_{i+1}$ in
$G/G_{i+1}.$ It is more difficult to compute the series $P_i(s) $
when $S_i$ is a non-abelian simple group. In that case an important
role is played by the group $L_i=G/C_G(G_i/G_{i+1}).$ This is a
monolithic primitive group and its unique minimal normal subgroup
is isomorphic to $G_i/G_{i+1}\cong S_i^{r_i}.$ If $n \neq
|S_i|^{r_i},$ then  the coefficient $b_{i,n}$ in (\ref{eq-pi}) depends only on $L_i;$ more precisely we have
\begin{equation*}
b_{i,n}=\sum_{\substack{|L_i:H|=n\\L_i=H \soc(L_i)}}\mu_{L_i}(H).
\end{equation*}

It is not easy to compute these coefficients $b_{i,n}$ even for  $n \neq
|S_i|^{r_i}.$
Some help comes from the knowledge of the subgroup $X_i$ of $\aut
S_i$ induced by the conjugation action of the normalizer in $L_i$ of
a composition factor of the socle $S_i^{r_i}$ (note that $X_i$ is an
almost simple group with socle isomorphic to $S_i).$
More precisely, given an almost simple group
$X$ with socle $S$, we can consider the following finite Dirichlet series:
\begin{equation}\label{pxs}P_{X,S}(s)=\sum_{n}\frac{c_n(X)}{n^s},~\textrm{\ where $c_n(X)=\sum_{\substack{|X:H|=n\\X=SH}}\mu_X(H)$.}
\end{equation}
\begin{lemma}\label{paz}\cite[Theorem 5]{Ser08}.
Let $S_i$ be a nonabelian  simple group and let  $\pi$ be a set of primes containing at least one divisor of $|S_i|$.
If n is not divisible by $\vert S_i\vert $ and $b_{i,n}\not= 0$, then there exists $m\in \Bbb N$ with $n=m^{r_i}$ and $b_{i,n}=c_m(X_i)\cdot m^{r_i-1}.$
This implies $$P_i^{\pi}(s)= P^{\pi}_{X_i,S_i}(r_is-r_i+1).$$
\end{lemma}

We will give now a description of the finite Dirichlet series $P^{\{p\}}_{X,S}(s)$ when $S$ is a simple group of Lie type over a field of characteristic $p$ and $X$ is an almost simple group with socle $S$. We follow the notations from \cite{Carter}.
Recall that a simple group of Lie type $S$ is the subgroup $A^F$ of fixed points under a Frobenius map $F$ of a connected reductive algebraic group $A$ defined over an algebraically closed field of characteristic $p>0$. In particular, $S$ is defined over a field $\bb{K}=\bb{F}_q$ of characteristic $p$. As explained in \cite[3.4]{Carter} a Dynkin diagram can be associated to the simple group $S$ and to the corresponding Lie algebra; moreover
(see \cite[13.3]{Carter}) to the map $F$, a symmetry $\rho$ on the Dynkin diagram of $A^F$ is associated ($\rho$ is trivial in the untwisted case). Let $I:=\{\mathcal{O}_1,\cdots,\mathcal{O}_k\}$ be the set of the $\rho$-orbits on the nodes of the Dynkin diagram. For every subset $J\subseteq I$, let $J^*:=\cup_{i\in J}\cal{O}_i$ be a $\rho$-stable subset of the set of nodes of the Dynkin diagram and one may associate an $F$-stable parabolic subgroup $P_J$ of $S$ with $J^*$. As described in {\cite[Chapter 9]{Carter}}, we may associate to $J$
a polynomial $T_{W_J}(x)$ with the property that $T_{W_J}(q)=|P_J|$. More precisely, in the notations of \cite[9.4]{Carter},
$T_{W_J}(x)=\sum_{w\in W_J}x^{l(w)}.$
 We have that:
\begin{theorem}\label{reduction of P_GX}\cite[Theorem 17]{Patassini}
Let $S$ be a simple group of Lie type defined over a field $\bb{K}=\bb{F}_q$ of characteristic $p$ and $X$ an almost simple group with socle $S$. 
Then $$P_{X,S}^{\{p\}}(s)=(-1)^{|I|}\sum_{J\subseteq I}(-1)^{|J|}\left(\frac{T_{W}(q)}{T_{W_J}(q)}\right)^{1-s}.$$
In particular, if $X$ does not contain non-trivial graph automorphisms, then  $$P^{\{p\}}_{X,S}(s)=P_S^{\{p\}}(s)$$
\end{theorem}

For later use we need to recall definitions and results concerning
Zsigmondy primes.
\begin{definition} Let $n\in \mathbb N$ with $n>1.$ A prime number $p$ is called a
{\em{primitive prime divisor}} of $a^n-1$ if it divides $a^n-1$
but it does not divide $a^e-1$ for any integer $1 \leq e \leq n-1$.
\end{definition}
The following theorem is due to  K. Zsigmondy \cite{Zsi}:
\begin{theorem}[Zsigmondy's Theorem]\label{zig}
 Let $a$ and $n$ be integers greater than 1. There exists a primitive
 prime divisor of  $a^n-1$ except exactly in the following cases:
\begin{enumerate}
\item $n=2,\, a=2^s-1$, where $s
\geq 2$. \item $n=6,\, a=2$.
\end{enumerate}
\end{theorem}
Observe that there may be more than one primitive prime divisor of
$a^n-1$; we denote by $\langle a,\, n \rangle$ the set of these
primes.

Let $p$ be a prime, $r$ a prime distinct from $p$ and $m$ an integer which is not a power of $p$. We define:
$$\begin{aligned}
\zeta_p(r)&=\min\{z\in \Bbb N \mid z\geq 1 \text{ and } p^z\equiv 1\!\!\!\! \mod r\},\\\zeta_p(m)&=\max\{\zeta_p(r) \mid r \text { prime }, r \neq p, r|m\}.
\end{aligned}$$
The value of $\zeta_p(S):=\zeta_p(|S|)$ when $S$ is a simple group of Lie type over $\bb{F}_q$ and $q=p^f$ is given in
in \cite[Table 5.2.C]{Lie90}.

\begin{proposition}\label{remark reduction 1}
Let $X$ be an almost simple group with socle $S$, where $S$ is a simple group of Lie type defined over a field of characteristic $p$. Assume that
$\zeta_p(S)>1$ and $\zeta_p(S)>6$ if $p=2.$ Let $\tau \in \l p,\zeta_p(S)\r$. Consider the Dirichlet series
$$P_{X,S}^{\{p\}}(s)=\sum_{n=1}^\infty \frac{a_n}{n^s}.$$
\begin{itemize}
\item[(a)] If $a_n\ne 0$ then $\tau$ divides $n$. More precisely, $v_{\tau}(n)=v_{\tau}(p^{\zeta_p(S)}-1)$.
\item[(b)] If $m>\zeta_p(S)$ and a primitive prime divisor of $p^m-1$ divides $n,$ then $a_n=0.$
\item[(c)] If $n$ is the smallest positive integer such that $n\neq 1$ and $a_{n}\ne 0,$ then $a_{n}<0$.
\end{itemize}
\end{proposition}
\begin{proof}
The difficult part of this proposition is (a). We use Theorem \ref{reduction of P_GX} and the description of the polynomials $T_{W}(t)$ and $T_{W_J}(t)$ given in \cite[Section 9.4, Section 14.2]{Carter} and in \cite[Section 3]{Patassini} (see in particular Table 1).
It turns out that if $q=p^f$, then $f$ divides $\zeta_p(S)$ and for each $J \subseteq I,$ the polynomial $T_{W_J}(t)$
can be written as a product of suitable cyclotomic polynomials $\Phi_u(t)$ with $u\leq \zeta_p(S)/f$. Moreover
$\Phi_{\zeta_p(S)/f}(t)$ appears with multiplicity exactly 1 in the factorization of $T_{W}(t)$ and does not divide $T_{W_J}(t)$
if $J\neq I.$ This means that if $\tau \in \l p,\zeta_p(S)\r$ and $J\neq I,$ then $v_{\tau}(T_W(q)/T_{W_J}(q))=
v_\tau(\Phi_{\zeta_p(S)/f}(p^f))=v_{\tau}(p^{\zeta_p(S)}-1).$
Now (b) follows from the fact that if $m>\zeta_p(S)$ then no prime divisor of $p^m-1$ divides $|S|.$ Finally (c) follows
from the fact that by the way in which $a_n$ is definited, the minimality of $n$ implies that the subgroups $H$ involved in the definition
of $a_n$ are maximal, thus $\mu_X(H)=-1$ and $a_n<0.$
\end{proof}

Combining the previous proposition with Lemma \ref{paz}, we obtain:
\begin{corollary}\label{remark reduction 2}
Assume that $G_i/G_{i+1}\cong S_i^{r_i}$ is a chief factors of $G$, where $S_i$ is a simple group of Lie type defined over a field of characteristic $p$. Assume that
$\zeta_p(S_i)>1$ and $\zeta_p(S_i)>6$ if $p=2.$  If $\tau\in \l p,\zeta_p(S_i)\r$, then we have:
\begin{itemize}
\item[(a)] If $b_{i,n}\ne 0$ and $(n,p)=1,$ then $\tau$ divides $n$. More precisely, $v_{\tau}(n)=r_i\cdot v_{\tau}(p^{\zeta_p(S_i)}-1)$.
\item[(b)] If $m>\zeta_p(S)$ and a primitive prime divisor of $p^m-1$ divides $n,$ then $b_{i,n}=0$.
\item[(c)] If $n$ is the smallest positive integer $>1$ such that $(n,p)=1$ and $b_{i,n}\ne 0,$ then $b_{i,n}<0$.
\end{itemize}
\end{corollary}
\section{Proofs of Theorem 1 and Theorem 2}
We start now the proofs of our main results.
We assume that $G$ is a finitely generated profinite group $G$  with the property that $P_G(s)=\sum_n a_n/n^s$ is rational.
As described in Section \ref{notation}, $P_G(s)$ can be written as a formal infinite product of finite Dirichlet series $P_i(s)
=\sum_{n\in \bb{N}}{b_{i,n}}/{n^s}$
corresponding to the factors $G_i/G_{i+1}$ of a chief series of $G.$
Let $J$ be the set of indices $i$ such that $G_i/G_{i+1}$ is a non-Frattini chief factor.
Since $P_i(s)=1$ if $i\notin J,$ we have
$$P_G(s)=\prod_{j\in J}P_j(s).$$

For $C(s)=\sum_{n=1}^\infty  c_n/n^s\in \cal R$, we define $\pi(C(s))$ to be the set of the primes $q$ for which there exists
at least one multiple $n$ of $q$ with $c_n\ne 0$. Notice that if $C(s)B(s)=A(s),$ then $\pi(C(s))\subseteq \pi(A(s))\bigcup \pi(B(s)).$
In particular if $C(s)$ is rational then $\pi(C(s))$ is finite.
Let $\cal S$ be the set of the finite simple groups that are isomorphic to a composition factor of some non-Frattini chief factor of $G$. The first step in the proofs of Theorem 1 and 2 is to show that $\cal S$ is finite. The proof of this claim requires the following result.
\begin{lemma}[{\cite[Lemma 3.1]{DeLu07}}]\label{number of non-iso composition factors}
Let $G$ be a finitely generated profinite group and let $q$ be a prime with $q\notin \pi(P_G(s))$.
Then $G$ has no maximal subgroup of index a power of $q.$ In particular if
 $q$ divides the order of a non-Frattini chief factor of $G,$ then this factor is not a $q$-group.
\end{lemma}

Let $\pi(G)$ be the set of the primes $q$ with the properties that $G$ contains at least an open subgroup $H$ whose index is divisible by $q$.
Obviously $\pi(P_G(s))\subseteq \pi(G)$. By \cite[Lemma 3.2]{DeLu07} and the classification of the finite simple groups, $\cal S$ is finite
if and only if $\pi(G)$ is finite.

\begin{lemma}\label{gamma finite}If $G$ satisfies the hypotheses  of either Theorem 1 or Theorem 2, then
the sets $\cal S$ and $\pi(G)$ are finite.
\end{lemma}
\begin{proof}
Since $P_G(s)$ is rational, we have that $\pi(P_G(s))$ is finite. Therefore, it follows from Lemma \ref{number of non-iso composition factors} that $\cal S$ contains only finitely many abelian groups. If $G$ satisfies the hypotheses of Theorem 2, then a non abelian group in  $\cal S$ is either one of the 26 sporadic simple groups or is isomorphic to a composition factor of the finite group $G/N$. In any case we have only finitely many possibilities. Consider now the case when $G$ satisfies the hypothesis of Theorem 1 and
assume by contradiction that $\cal S$ is infinite. This is possible only if the
subset $\cal S^*$ of  the simple groups in $\cal S$ that are of Lie type over a field of characteristic $p$ is infinite.
In particular, the set $\Omega=\{\zeta_p(S)\mid S \in \cal S^*\}$ is infinite (see \cite[Table 5.2C]{Lie90}).
Let $$I:=\{j\in J\mid S_j \in \cal S^*\},
\ A(s):=\prod_{i\in I}P_i(s)\ \text{ and }\ B(s):=\prod_{i\notin I}P_i(s).$$ Notice that $\pi(B(s))\subseteq \bigcup_{S\in \cal S\setminus \cal S^*}\pi(S)$ is a finite set. Since $P_G(s)=A(s)B(s)$ and $\pi(P_G(s))$ is finite, if follows that the set $\pi(A(s))$ is finite.
According with Theorem \ref{zig}, if $m$ is large enough (for example if $m > 6)$  then the set $\l p,m\r$ is non empty.
We can find a  positive integer $m\in\Omega$ such that $\l p,m\r\neq \emptyset$ but $\l p,m\r \cap \pi(A(s))=\emptyset.$
Notice that if $m\neq u$ then $\l p,m \r \cap \l p, u \r=\emptyset.$
Let $\Gamma_m$ be the set of the positive integers $n$ such that  there exists $\tau \in \l p, m \r$ dividing $n$ but no prime in
$\l p,u\r$ divides $n$ if $u>m$. Notice that if $b_{i,n}\neq 0,$ then $\zeta_p(S_i)=m$ if and only if $n\in \Gamma_m.$
Set
$$
\begin{aligned}
&r:=\min\{r_i \mid S_i\in\cal{S}^* \text{ and } \zeta_p(S_i)=m\},\\&I^*:=\{i\in I \mid r_i=r~\rm{and}~S_i\in\cal{S}\},\\
&\beta:=\min\{n>1 \mid n \in \Gamma_m \text{ and }b_{i,n} \neq 0 \text { for some } i \in I^*\}.
\end{aligned}$$
By Corollary \ref{remark reduction 2},
if $i\in I$ and $b_{i,\beta} \neq 0$, then $\zeta_p(S_i)=m,$ $r_i=r$ and
$b_{i,\b}<0.$
Hence the coefficient $c_{\b}$ of $1/{\b^s}$ in $A(s)$ is
$$c_\b=\sum_{i\in I,r=r_i}b_{i,\b}=\sum_{i\in I^*}b_{i,\b}<0.$$
On the other hand, again by  Corollary \ref{remark reduction 2}, all the primes in $\l p,m\r$ divide $\beta.$
But then $\l p,m\r\subseteq \pi(A(s)),$ which is a contradiction. So we have proved that $\cal S$ is finite. By \cite[Lemma 3.2]{DeLu07}, if follows that $\pi(G)$ is also finite.
\end{proof}



The previous result allows us to employ the following:

\begin{proposition}\label{two conditions fulfilled}
Let $G$ be a finitely generated profinite group and assume that $\pi(G)$ is finite. For each $n$, there are only finitely many non-Frattini factors in a chief series whose composition length is at most $n$. Moreover   there exists a prime $t$ such that no non-Frattini chief factor of $G$ has
composition length divisible by $t.$
\end{proposition}

\begin{proof}Since $\pi(G)$ is finite, the set $\cal S$ of the composition factors of $G$ is also finite and therefore there exists
$u\in \mathbb N$ such that $|S|\leq u$ for each $S \in \cal S.$ Now assume by contradiction that a chief series of $G$ contains infinitely many non-Frattini
chief factors of composition length at most $n$. Let $X/Y$ be one of them: since $X/Y$ is non-Frattini there exists a proper supplement $H/Y$ of $X/Y$ in
$G/Y$. Clearly $|G:H|\leq |X/Y| \leq u^n$. In this way we construct infinitely many subgroups of index at most $u^n$, which is not possible since a finitely
generated profinite group contains only finitely many subgroups of a given index. The second part of the statement is \cite[Corollary 5.2]{DeLu07}.
\end{proof}


For a simple group $S\in\cal {S}$, let $I_S=\{j\in J \mid S_j \cong S\}$.
Our aim is to prove that, under the hypotheses of Theorem 1 and 2,  $J$ is a finite set. We have already proved that $\cal S$ is finite, so it suffices to prove that $I_S$ is finite for each $S \in \cal S.$ First we consider the case when $S$ is abelian.

\begin{lemma}\label{abelian finite}Assume that $G$ is a finitely generated profinite group
such that $P_G(s)$ is rational and $\pi(G)$ is finite. Then for any prime $q$, if $S$ has a subgroup
with index a power of $q,$ then $I_S$ is finite. In particular if
$S$ is cyclic, then $I_S$ is finite.
\end{lemma} %
\begin{proof}
Let $\cal{S}_q$ be the set of the non abelian simple groups in $\cal {S}$ containing a proper subgroup  of $q$-power index.
A theorem proved by Guralnick \cite{gur} implies that if $T\in \cal{S}_q$ then there exists a unique
positive integer $\alpha(T)$ with the property that $T$ contains a subgroup of index $q^{\alpha(T)}.$
Consider the set $\pi$ of all the primes different from $q$. By
Lemma \ref{paz},
there exist positive integers $c_i$
and nonnegative integers $d_i$ such that
\begin{equation}\label{prodoinfo}P_{G}^\pi(s)=\prod_{i\in I_S}\left(1-\frac{c_i}{q^{r_is}}\right)\prod_{T\in\cal{S}_q}\left(\prod_{j\in I_T}\left(1-\frac{d_j}{q^{\alpha(T)r_js}}\right)\right)
\end{equation}

Since $\cal{S}$ is finite,  the set $\{\alpha(T) \mid T\in\cal{S}_q\}$ is finite. Moreover, by Proposition \ref{two conditions fulfilled}, there is a prime number $t$ such that no element in $$\{r_i \mid i\in I_S\}\bigcup\{\alpha(T)r_j \mid T\in \cal{S}_q \text { and } j\in  I_T\}$$ is divisible by $t$. Since $P_G(s)$ is rational, $P_{G}^\pi(s)$ is also rational. But then, by Proposition \ref{skolem theorem}, the number of nontrivial factors in the product at the right side of equation (\ref{prodoinfo}) is finite. In particular, $I_S$ is a finite set.
\end{proof}

\begin{proof}[{Proof of Theorem 1}]
Let $\cal T$ be the set of the almost simple groups $X$ such that
there exist infinitely many $i\in J$ with $X_i\cong X$
and let $I=\{i\in J\mid X_i \in \cal T\}.$
The hypotheses of Theorem 1 combined with Lemma \ref{abelian finite} implies that $J\setminus I$ is finite.
We have to prove that $J$ is finite; this is equivalent to show that $I=\emptyset.$
But then, in order to complete our proof, it suffices to prove  the following claim.

\noindent $(*)$ {\sl For every $n\in\Bbb N,$ $I_n=\{i\in I \mid \zeta_p(S_i)=n\}=\emptyset.$ }

\noindent Assume that the claim is false and let $m$ be the smallest integer such that the set $I_m\neq \emptyset.$
Since $J\setminus I$ is finite and $P_G(s)=\prod_{i\in J}P_i(s)$ is rational,
also $\prod_{i\in I}P_i(s)$ is rational. In particular, the following series is rational:
$$Q(s)=\prod_{i\in I}P^{\{p\}}_i(s).$$
We distinguish three different cases:
\begin{itemize}
  \item[(1)] $m=1, p=2^t-1, t\geq 2;$
  \item[(2)] $m\leq 5, p=2;$
  \item[(3)] all the other possibilities.
 \end{itemize}
Assume that $(1)$ occurs. By \cite[Table 5.2.C]{Lie90} if $\zeta_p(S)=1,$ then $S\cong PSL_2(p)$. In particular
$S$ has a subgroup of index a power of 2 and $I_S$ (and consequently $I_1$) is finite by Lemma \ref{abelian finite}.

In case  (3), it follows by Theorem \ref{zig} that $\l p,t\r\neq \emptyset $  for every $t > m;$ we set $\pi=\bigcup_{t>m}\l p,t\r.$ In case (2),  $\l p,t\r\neq \emptyset$   whenever $t>6$ and we set $\pi=\bigcup_{t>6}\l p,t\r.$ The Dirichlet series
$H(s)=Q^{\pi}(s)$ is rational.  By Corollary \ref{remark reduction 2}, if $i\in I_t$ and $\tau \in \l p,t\r,$ then
$P_i^{\{\tau,p\}}(s)=1$; in particular $P_i^{\pi}(s)=1$ whenever $\l p,t\r \subseteq \pi.$ This implies
$$H(s)=\begin{cases}\prod_{i\in I_m}P^{\{p\}}_i(s)& \text{ in case (3) },\\\prod_{\substack{i\in I_u\\ m\leq u\leq 5}}P^{\{2\}}_i(s)
&{\text { in case (2).}}
\end{cases}$$

Assume that case $(3)$ occurs and let $\tau\in \l p,m\r.$
By Lemma \ref{paz} and Corollary \ref{remark reduction 2}, if $i\in I_m$, $(p,y)=1$ and
$b_{i,y}\neq 0,$ then $y=x^{r_i}$ and $v_{\tau}(x)=v_\tau(p^m-1).$
Let
$$w=\min\{x\in\bb{N} \mid v_{\tau}(x)=v_{\tau}(p^m-1)~\rm{and}~b_{i,x^{r_i}}\ne 0~\rm{for some $i\in I_m$}\}.$$
By Corollary \ref{remark reduction 2}, for each $i\in I_m$, if $b_{i,w^{r_i}}\ne 0$ then $b_{i,w^{r_i}}<0$.
Moreover if $b_{i,w^{r_i}}\ne 0$ and $X_j\cong X_i,$ then $b_{j,w^{r_j}}\ne 0$, so the set
$\Sigma_m=\{i\in I_m \mid b_{i,w^{r_i}}\ne 0\}$ is infinite.
Applying Proposition \ref{prop 4.3}, we obtain a rational product
$$
H^*(s)=\prod_{i \in \Sigma_m}\left(1+\frac{b_{i,w^{r_i}}}{w^{r_is}}\right),~\rm{where $b_{i,w^{r_is}}<0$~ for all $i\in \Sigma_m$}.
$$
By Propositions \ref{skolem theorem} and \ref{two conditions fulfilled},  $H^*(s)$ is a finite product, i.e. $\Sigma_m$ is finite, which is a contradiction.


Finally assume that case $(2)$ occurs. If $\zeta_p(S)\leq 5,$ then $S$ is one of the following  groups:
$PSL_6(2), U_4(2), PSp_6(2), P\Omega^+_8(2), PSL_3(4), SL_5(2), PSL_4(2), PSL_3(2).$
The explicit description of the Dirichlet series $P_{X,S}^{\{2\}}(s)$ when $S\leq X\leq \aut(S)$ and $S$ is one of the simple groups in the previous list is included in the Appendix. Notice in particular that if $i\in \Lambda=\bigcup_{m\leq 5}I_m$
then $\pi(P^{\{2\}}_i(s)) \subseteq \{3,7,5,31\}.$ First consider $$\Lambda_{31}=\{i\in \Lambda \mid 31 \in \pi(P^{\{2\}}_i(s))\}$$
and let
\begin{eqnarray*}
w & = & \min\{x\in\bb{N}\mid x \text { is odd}, v_{31}(x)=1~\rm{and}~b_{i,x^{r_i}}\ne 0~\rm{for some $i\in \Lambda$}\} \\
  & = &  \min\{x\in\bb{N}\mid x \text { is odd}, v_{31}(x)=1~\rm{and}~b_{i,x^{r_i}}\ne 0~\rm{for some $i\in \Lambda_{31}$}\}.
\end{eqnarray*}
Note that if $i\in\Lambda_{31}$ and $n$ is minimal with the properties that $n$ is odd, $b_{i,n^{r_i}}\ne 0$ and $v_{31}(n)=1$, then $b_{i,n^{r_i}}<0$ (see Appendix). So if $b_{i,w^{r_i}}\ne 0$ then $b_{i,w^{r_i}}<0$; moreover,
by applying Proposition \ref{prop 4.3}, we obtain a rational product
$$
H^*(s)=\prod_{i \in \Lambda}\left(1+\frac{b_{i,w^{r_i}}}{w^{r_is}}\right)=\prod_{i \in \Lambda_{31}}\left(1+\frac{b_{i,w^{r_i}}}{w^{r_is}}\right)
,~\rm{where $b_{i,w^{r_i}}\leq 0$~ for all $i\in \Lambda_{31}$}.
$$
By Propositions \ref{skolem theorem} and \ref{two conditions fulfilled},
the set $\Lambda^*_{31}=\{i\in \Lambda_{31}\mid b_{i,w^{r_i}}\neq 0\}$ is finite, but
this implies that $\Lambda_{31}=\emptyset$. Indeed if $\Lambda_{31}\neq \emptyset$ then
there exists at least one index $i$ with $i\in \Lambda^*_{31}$, moreover by assumption there are infinitely many $j$ with $X_j\cong X_i$ and all of them belong to $\Lambda^*_{31}$. Since $\Lambda_{31}=\emptyset,$
 if $i\in \Lambda$, then $S_i$ is isomorphic to one of the following:
$U_4(2), PSp_6(2), P\Omega^+_8(2), PSL_3(4), PSL_4(2), PSL_3(2)$. It follows from Appendix, that if $i\in \Lambda$,
$x$ is odd and $b_{i,x^{r_i}}\neq 0$, then $v_7(x)\leq 1.$ But then, we may repeat the same argument as above and
consider $\Lambda_{7}=\{i\in \Lambda \mid 7 \in \pi(P^{\{2\}}_i(s))\}$
and
$$
w:= \min\{x\in\bb{N}\mid x \text { is odd}, v_7(x)=1~\rm{and}~b_{i,x^{r_i}}\ne 0~\rm{for some $i\in \Lambda_7$}\}.$$
Arguing as before we deduce that $\Lambda_7=\emptyset.$
We can see from Appendix  that this implies  $S_i\cong U_4(2)$ for all $i\in \Lambda$ and
$$H^{\{5\}}(s)=\prod_{i\in\Lambda}\left(1-\frac{3^{3r_i}}{3^{3r_is}}\right).$$
Again, by Propositions \ref{skolem theorem} and \ref{two conditions fulfilled},  $\Lambda$ is finite and consequently  $\Lambda=\emptyset.$
\end{proof}

\begin{proof}[{Proof of Theorem 2}]
Let $\cal T$ be the set of the almost simple groups $X$ such that
$\soc X$ is a sporadic simple groups and there exist infinitely many $i\in J$ with $X_i\cong X$
and let $I=\{i\in J\mid X_i \in \cal T\}.$
As in the case of Theorem 1,
we have to prove that $I=\emptyset.$ For an almost simple group $X$, let $\Omega(X)$ be the set of the odd integers
$m\in \Bbb N$ such that
\begin{itemize}
\item $X$ contains at least one subgroup $Y$ such that $X=Y\soc X$ and $|X:Y|=m;$
\item if $X=Y\soc X $ and $|X:Y|=m,$ then $Y$ is a maximal subgroup if $X.$
\end{itemize}
Note that if $m \in \Omega(X),$ $X=Y\soc X $ and $|X:Y|=m,$
then $\mu_X(Y)=-1$: in particular
$c_m(X)<0.$ Combined with Lemma \ref{paz}, this implies that if $m\in \Omega(X_i)$ then $b_{i,m^{r_i}}<0.$
Certainly $\Omega(X)$ is not empty and its smallest element is the smallest index $m(X)$ of a supplement of
$\soc X$ in $X$ containing a Sylow 2-subgroup of $X.$ When $S=\soc X$ is a sporadic simple group, the value of $m(X)$
can be read from \cite{atlas}, where, for each of these groups, the list of the maximal subgroups and their indices are given; the precise values are given in Table \ref{prima}. In few cases
we need to know another integer $n(X)$ in $\Omega(X),$ given in Table \ref{seconda}. For a fixed prime $p$, let
$\Lambda_{p}=\{i\in I \mid p \in \pi(P^{\{2\}}_i(s))\}.$

If $i\in \Lambda_{31},$ then $31$ divide $|S_i|$ and $S_i \in \{\rm{J}_4, \rm{Ly}, \rm{O'N}, \rm{BM}, \rm{M},\rm{Th}\}.$
Moreover $31^2$ does not divide $|S_i|$ so if $n$ is odd, divisible by 31 and $b_{i,n}\neq 0$
then $n=x^{r_i}$ and $v_{31}(x)=1.$ Let $m_i=n(S_i)$ if $S_i\cong \rm{Th},$ $m_i=m(S_i)$ otherwise.
Since $m_i$ is the smallest odd number divisible
by $31$ and equal to the index in $X_i$ of a supplement of $S_i$ we get:
\begin{eqnarray*}
w & = &  \min \{ x\in\bb{N}\mid x \text { is odd }, v_{31}(x)=1~\rm{and}~b_{i,x^{r_i}}\ne 0~\rm{for some $i\in I$} \} \\
 & = & \min \{ x\in\bb{N}\mid x \text { is odd }, v_{31}(x)=1~\rm{and}~b_{i,x^{r_i}}\ne 0~\rm{for some $i\in \Lambda_{31}$} \}\\
& = & \min \{m_i \mid i \in \Lambda_{31}\}.
\end{eqnarray*}
But then by Proposition \ref{prop 4.3}, the following Dirichlet series is rational:
$$\prod_{i\in\Lambda_{31}}\left(1+\frac{b_{i,w^{r_i}}}{(w^{r_i})^s}\right).$$
We have $b_{i,w^{r_i}}<0$ if $m_i=w,$ $b_{i,w^{r_i}}=0$ otherwise.
By applying Propositions \ref{skolem theorem} and \ref{two conditions fulfilled}, we get that $\{i \in \Lambda_{31}\mid m_i=w\}$ is a finite set,
and this implies $\Lambda_{31}=\emptyset.$

Now consider  $\Lambda_{23}$. Since $\Lambda_{31}=\emptyset$,
if $i\in\Lambda_{23}$ then $S_i\in \{\rm{M}_{23}$, $\rm{M}_{24}$, $\rm{Co}_1$, $\rm{Co}_2$, $\rm{Co}_3$, $\rm{Fi}_{23}$, ${\rm{Fi}_{24}}^\prime\}$.
We can repeat the argument used to proved that $\Lambda_{31}=\emptyset.$ Let $m_i=n(S_i)$ if $S_i\cong \rm{Co}_1,$ $m_i=m(X_i)$ otherwise
and let $w = \min \{m_i \mid i \in \Lambda_{23}\}.$ By applying Propositions \ref{skolem theorem} and \ref{two conditions fulfilled}, we get that $\{i \in \Lambda_{23}\mid m_i=w\}$ is a finite set, and this implies $\Lambda_{23}=\emptyset.$

New we consider $\Lambda_{11}.$ Since $\Lambda_{31}\cup \Lambda_{23}=\emptyset,$
if $i\in\Lambda_{11}$, then $S_i\in \{{\rm{M}_{11}}$, $\rm{M}_{12}$, $\rm{M}_{22}$, $\rm{J}_1$, $\rm{HS}$, $\rm{Suz}$, $\rm{McL}$, $\rm{HN}$, $\rm{Fi}_{22}\}$. Let $m_i=n(X_i)$ if $S_i\cong \rm{Fi}_{22}$ or $S_i\cong \rm{Fi}^\prime_{24},$ $m_i=m(X_i)$ otherwise
and let $w = \min \{m_i \mid i \in \Lambda_{11}\}.$ As before, by applying Propositions \ref{skolem theorem} and \ref{two conditions fulfilled}, we get that $\{i \in \Lambda_{23}\mid m_i=w\}$ is a finite set, and this implies $\Lambda_{11}=\emptyset.$
Continuing our procedure, we consider $\Lambda_{17}$:
if $i\in\Lambda_{17}$,
then $S_i\in \{\rm{J}_3, \rm{He}\}$ and we can take
$w = \min \{m(X_i) \mid i \in \Lambda_{17}\}$ and  deduce that $\Lambda_{17}=\emptyset.$
Next we take $w = m(\rm{Ru})$ to prove $\Lambda_{29}=\emptyset$
and finally we take $w = m(\rm{J}_2)$ to prove $\Lambda_{7}=\emptyset.$
\end{proof}

\section*{Appendix: exceptional cases}
\addcontentsline{toc}{section}{Appendix: exceptional cases}
In this section, we give explicit formulae for $P_{X,S}^{(2)}(s)$ when $X$ is an almost simple group
whose socle $S$ is of Lie type over a field of characteristic $2$ and $\zeta_2(S)\leq 5$.
We use for this purpose Theorem \ref{reduction of P_GX} and the description of the polynomials $T_{W}(t)$ and $T_{W_J}(t)$ given in \cite[Section 9.4, Section 14.2]{Carter} and in \cite[Section 3]{Patassini} (see in particular Table 1). Just to see an example, consider the case $S=\rm{PSL}_4(2).$
$I=\{1,2,3\}$ is the set of the nodes of the Dynkin diagram. According to \cite[Table 1]{Patassini}
$T_W(2)=(2^4-1)(2^3-1)(2^2-1)=3^2\cdot 5\cdot 7$. Moreover
$T_{W_J}(2)=(2^3-1)(2^2-1)=3\cdot 7$ if $J\in \{\{1,2\},\{2,3\}\},$ $T_{W_J}(2)=(2^2-1)(2^2-1)=3^2$ if $J=\{1,3\},$
$T_{W_J}(2)=(2^2-1)=3$ if $J$ is one of the three subsets of $I$ of cardinality 1, $T_{W_\emptyset}(2)=1.$
Hence
$P_{S}^{(2)}(s)=1-2(3\cdot5)^{(1-s)}-(5\cdot7)^{(1-s)}+3(3\cdot5\cdot7)^{(1-s)}-(3^2\cdot5\cdot7)^{(1-s)}.$

\begin{itemize}
  \item[(i)] $S=\rm{PSL}_6(2)$. If $X$ contains a graph automorphism then
\begin{eqnarray*}
P_{X,S}^{(2)}(s)&=&1-(3^2\cdot7\cdot31)^{(1-s)}-(3\cdot5\cdot7^2\cdot31)^{(1-s)}-(3^3\cdot7\cdot31)^{(1-s)}+ \\
&&+ \ 2(3^4\cdot7^2\cdot31)^{(1-s)}+(3^3\cdot5\cdot7^2\cdot31)^{(1-s)}-(3^4\cdot5\cdot7^2\cdot31)^{(1-s)}.
\end{eqnarray*}
If $X$ does not contain graph automorphisms, then
\begin{eqnarray*}
P_{X,S}^{(2)}(s)&=&1-2(3^2\cdot7)^{(1-s)}-(3^2\cdot5\cdot31)^{(1-s)}-2(3\cdot7\cdot31)^{(1-s)}+ \\
&&+\ 3(3^2\cdot7\cdot31)^{(1-s)}+6(3^2\cdot5\cdot7\cdot31)^{(1-s)}+(3\cdot5\cdot7^2\cdot31)^{(1-s)}-\\
&&-\ 4(3^3\cdot5\cdot7\cdot31)^{(1-s)}-6(3^2\cdot5\cdot7^2\cdot31)^{(1-s)}+\\
&&+\ 5(3^3\cdot5\cdot7^2\cdot31)^{(1-s)}-(3^4\cdot5\cdot7^2\cdot31)^{(1-s)}.
\end{eqnarray*}
  \item[(ii)] $S=\rm{PSL}_5(2)$. If $X$ contains a graph automorphism then
\begin{eqnarray*}
P_{X,S}^{(2)}(s)&=&1-(3\cdot5\cdot31)^{(1-s)}-(3^2\cdot7\cdot31)^{(1-s)}+(3^2\cdot5\cdot7\cdot31)^{(1-s)}.
\end{eqnarray*}
If $X$ does not contain graph automorphisms, then
\begin{eqnarray*}
P_{X,S}^{(2)}(s)&=&1-2(31)^{(1-s)}-2(5\cdot31)^{(1-s)}+3(3\cdot5\cdot31)^{(1-s)}+\\
&&+ \ 3(5\cdot7\cdot31)^{(1-s)} -4(3\cdot5\cdot7\cdot31)^{(1-s)}+(3^2\cdot5\cdot7\cdot31)^{(1-s)}.
\end{eqnarray*}
  \item[(iii)] $S=\rm{PSL}_4(2)$. If $X$ contains a graph automorphism then
\begin{eqnarray*}
P_{X,S}^{(2)}(s)&=&1-(3^2\cdot7)^{(1-s)}-(3\cdot5\cdot7)^{(1-s)}+(3^2\cdot5\cdot7)^{(1-s)}.
\end{eqnarray*}
If $X$ does not contain graph automorphisms, then
\begin{eqnarray*}
P_{X,S}^{(2)}(s)&=&1-2(3\cdot5)^{(1-s)}-(5\cdot7)^{(1-s)}+3(3\cdot5\cdot7)^{(1-s)}-(3^2\cdot5\cdot7)^{(1-s)}.
\end{eqnarray*}
  \item[(iv)] $S=\rm{PSL}_3(2)$. If $X$ contains a graph automorphism then
\begin{eqnarray*}
P_{X,S}^{(2)}(s)&=&1-(3\cdot7)^{(1-s)}.
\end{eqnarray*}
If $X$ does not contain graph automorphisms, then
\begin{eqnarray*}
P_{X,S}^{(2)}(s)&=&1-2(7)^{(1-s)}+(3\cdot7)^{(1-s)}.
\end{eqnarray*}
  \item[(v)] $S=\rm{PSL}_3(4)$. If $X$ contains a graph automorphism then
  $$P_{X,S}^{(2)}(s)=1-(3\cdot5\cdot7)^{(1-s)}.$$
If $X$ does not contain graph automorphisms then
$$P_{X,S}^{(2)}(s)=1-2(3\cdot7)^{(1-s)}+(3\cdot5\cdot7)^{(1-s)}.$$
  \item[(vi)] $S=\rm{PSp}_6(2)$.
   We have
  $$P_{X,S}^{(2)}(s)=1-(3^2\cdot7)^{(1-s)}-(3^3\cdot5)^{(1-s)}-(3^2\cdot5\cdot7)^{(1-s)}+3(3^3\cdot5\cdot7)^{(1-s)}-(3^4\cdot5\cdot7)^{(1-s)}.$$
  \item[(vii)] $S=\rm{U}_4(2)$. We have
  $$P_{X,S}^{(2)}(s)=1-(3^3)^{(1-s)}-(3^2\cdot5)^{(1-s)}+(3^3\cdot5)^{(1-s)}.$$
  \item[(viii)] $S=\rm{P}\Omega^+_8(2)$. We have
\begin{eqnarray*}
P_{X,S}^{(2)}(s) & = & 1-3(3^2\cdot5)^{(1-s)}-(3\cdot5^2\cdot7)^{(1-s)}+3(3^3\cdot5^2)^{(1-s)}+ \\
&&+ \ 3(3^3\cdot5^2\cdot7)^{(1-s)}-4(3^4\cdot5^2\cdot7)^{(1-s)}+(3^5\cdot5^2\cdot7)^{(1-s)}.\end{eqnarray*}

\end{itemize}

\begin{longtable}{|l|l|l|}
\caption{Sporadic simple groups}\label{prima}\\
\hline
$X$ & $|X|$ & $m(X)$\\
\hline
{$\rm{M}_{11}$} & {$2^4\cdot3^2\cdot5\cdot11$} & $11$ \\
\hline
{$\rm{M}_{12}$}  & {$2^6\cdot3^3\cdot5\cdot11$} & $3^2\cdot5\cdot11$ \\
\hline
{$\aut(\rm{M}_{12})$}  & {$2^7\cdot3^3\cdot5\cdot11$} & $3^2\cdot5\cdot11$ \\
\hline
{$\rm{M}_{22}$}  & {$2^7\cdot3^2\cdot5\cdot7\cdot11$} & $7\cdot11$  \\
\hline
{$\aut(M_{22})$}  & {$2^8\cdot3^2\cdot5\cdot7\cdot11$} & $7\cdot11$  \\
\hline
{$\rm{M}_{23}$}  & {$2^7\cdot3^2\cdot5\cdot7\cdot11\cdot23$}  &  $23$  \\
\hline
{$\rm{M}_{24}$}  &  {$2^{10}\cdot3^3\cdot5\cdot7\cdot11\cdot23$} &  $3\cdot11\cdot23$  \\
\hline
{$\rm{J}_1$}  & {$2^3\cdot3\cdot5\cdot7\cdot11\cdot19$} &  $5\cdot11\cdot19$ \\
\hline
{$\rm{J}_2$}  & {$2^7\cdot3^3\cdot5^2\cdot7$} & $3^2\cdot5\cdot7$ \\
\hline
{$\aut(\rm{J}_2)$}  & {$2^8\cdot3^3\cdot5^2\cdot7$} & $3^2\cdot5\cdot7$ \\
\hline
{$\rm{J}_3$} & {$2^7\cdot3^5\cdot5\cdot17\cdot19$} & $3^4\cdot17\cdot19$\\
\hline
{$\aut(\rm{J}_3)$} & {$2^8\cdot3^5\cdot5\cdot17\cdot19$} & $3^4\cdot17\cdot19$\\
\hline
{$\rm{J}_4$}  & {$2^{21}\cdot3^3\cdot5\cdot7\cdot11^3\cdot23\cdot29\cdot31\cdot37\cdot43$} & $11^2\cdot29\cdot31\cdot37\cdot43$\\
\hline
{$\rm{HS}$}  & {$2^9\cdot3^2\cdot5^3\cdot7\cdot11$} & $3\cdot5^3\cdot11$\\
\hline
{$\aut(\rm{HS})$}  & {$2^{10}\cdot3^2\cdot5^3\cdot7\cdot11$} & $3\cdot5^3\cdot11$\\
\hline
{$\rm{Suz}$} & {$2^{13}\cdot3^7\cdot5^2\cdot7\cdot11\cdot13$} & $3^3\cdot5\cdot7\cdot11\cdot13$ \\
\hline
{$\aut(\rm{Suz})$} & {$2^{14}\cdot3^7\cdot5^2\cdot7\cdot11\cdot13$} & $3^3\cdot5\cdot7\cdot11\cdot13$ \\
\hline
{$\rm{McL}$}  & {$2^7\cdot3^6\cdot5^3\cdot7\cdot11$} & $5^2\cdot11$  \\
\hline
{$\aut(\rm{McL}$)}  & {$2^8\cdot3^6\cdot5^3\cdot7\cdot11$} & $5^2\cdot11$  \\
\hline
{$\rm{Ru}$}  & {$2^{14}\cdot3^3\cdot5^3\cdot7\cdot13\cdot29$} & $3^2\cdot5^3\cdot13\cdot29$  \\
\hline
{$\rm{He}$}  & {$2^{10}\cdot3^3\cdot5^2\cdot7^3\cdot17$} & $5\cdot7^3\cdot17$  \\
\hline
{$\aut(\rm{He})$}  & {$2^{11}\cdot3^3\cdot5^2\cdot7^3\cdot17$} & $3^2\cdot5^2\cdot7^2\cdot17$  \\
\hline
{$\rm{Ly}$}  & {$2^8\cdot3^7\cdot5^6\cdot7\cdot11\cdot31\cdot37\cdot67$} &  $5^3\cdot31\cdot37\cdot67$  \\
\hline
{$\rm{O'N}$} & {$2^9\cdot3^4\cdot5\cdot7^3\cdot11\cdot19\cdot31$} & $3^2\cdot7^2\cdot11\cdot19\cdot31$ \\
\hline
{$\aut(\rm{O'N})$} & {$2^{10}\cdot3^4\cdot5\cdot7^3\cdot11\cdot19\cdot31$} & $3^2\cdot7^2\cdot11\cdot19\cdot31$ \\
\hline
{$\rm{Co}_1$}  & {$2^{21}\cdot3^9\cdot5^4\cdot7^2\cdot11\cdot13\cdot23$} & $3^6\cdot5^3\cdot7\cdot13$  \\
\hline
{$\rm{Co}_2$}  & {$2^{18}\cdot3^6\cdot5^3\cdot7\cdot11\cdot23$} & $3^4\cdot5^2\cdot23$   \\
\hline
{$\rm{Co}_3$}  & {$2^{10}\cdot3^7\cdot5^3\cdot7\cdot11\cdot23$} & $3^3\cdot5^2\cdot11\cdot23$  \\
\hline
{$\rm{Fi}_{22}$}  & {$2^{17}\cdot3^9\cdot5^2\cdot7\cdot11\cdot13$} & $3^7\cdot5\cdot13$  \\
\hline
{$\aut(\rm{Fi}_{22})$}  & {$2^{18}\cdot3^9\cdot5^2\cdot7\cdot11\cdot13$} & $3^7\cdot5\cdot13$  \\
\hline
{$\rm{Fi}_{23}$}  & {$2^{18}\cdot3^{13}\cdot5^2\cdot7\cdot11\cdot13\cdot17\cdot23$} & $3^4\cdot17\cdot23$ \\
\hline
{$\rm{Fi}^\prime_{24}$}  & {$2^{21}\cdot3^{16}\cdot5^2\cdot7^3\cdot11\cdot13\cdot17\cdot23\cdot29$} & $3^{13}\cdot5\cdot7^2\cdot13\cdot17\cdot29$\\
\hline
{$\aut(\rm{Fi}^\prime_{24})$}  & {$2^{22}\cdot3^{16}\cdot5^2\cdot7^3\cdot11\cdot13\cdot17\cdot29$} & $3^{13}\cdot5\cdot7^2\cdot13\cdot17\cdot29$\\
\hline
{$\rm{HN}$} & {$2^{14}\cdot3^6\cdot5^6\cdot7\cdot11\cdot19$} & $3^4\cdot5^4\cdot7\cdot11\cdot19$ \\
\hline
{$\aut(\rm{HN})$} & {$2^{15}\cdot3^6\cdot5^6\cdot7\cdot11\cdot19$} & $3^4\cdot5^4\cdot7\cdot11\cdot19$ \\
\hline
{$\rm{Th}$}  & {$2^{15}\cdot3^{10}\cdot5^3\cdot7^2\cdot13\cdot19\cdot31$} & $3^8\cdot5^2\cdot7\cdot13\cdot19$\\
\hline
{$\rm{BM}$} & {$2^{41}\cdot3^{13}\cdot5^6\cdot7^2\cdot11\cdot13\cdot17\cdot19\cdot23\cdot31\cdot47$} & $3^7\cdot5^3\cdot7\cdot13\cdot 17\cdot19\cdot31\cdot47$ \\
\hline
\multirow{2}{*}{$\rm{M}$}  & {$2^{46}\cdot3^{20}\cdot5^9\cdot7^6\cdot11^2\cdot13^3\cdot17\cdot19\cdot23\cdot29$} &
{$3^{11}\cdot5^5\cdot7^4\cdot11\cdot13^2\cdot17\cdot 19\cdot$}\\
& {$\cdot31\cdot41\cdot 47\cdot59\cdot71$} & {$\cdot29\cdot31\cdot41\cdot47\cdot59\cdot71$}\\
\hline
\end{longtable}

\begin{longtable}{|l|l|l|}
\caption{$n(X)$}\label{seconda}\\
\hline
$X$ &  $n(X)$\\
\hline
{$\rm{Co}_1$}  & {$3^4\cdot5^2\cdot7\cdot11\cdot13\cdot23$}   \\
\hline
{$\rm{Fi}_{22}$}  & {$3^5\cdot5\cdot7\cdot11\cdot13$}  \\
\hline
{$\rm{Fi}^\prime_{24}$}  &  $3^{9}\cdot5\cdot 11\cdot7^2\cdot13\cdot17\cdot 23\cdot29$\\
\hline
{$\aut(\rm{Fi}^\prime_{24})$}   & $3^{9}\cdot5\cdot 11\cdot7^2\cdot13\cdot17\cdot 23\cdot29$\\
\hline
{$\rm{Th}$}  & $3^8\cdot 5^2\cdot 7\cdot 13 \cdot 19\cdot 31$\\
\hline
\end{longtable}

\bibliographystyle{alpha}

\addcontentsline{toc}{section}{References}
\end{document}